\documentclass[11pt]{amsart}

\pdfoutput=1

\usepackage[utf8]{inputenc}
\usepackage[T1]{fontenc}
\usepackage{lmodern}
\usepackage{amsthm, amssymb, amsmath, amsfonts, mathrsfs}

\usepackage{microtype}

\usepackage[pagebackref,colorlinks=true,pdfpagemode=none,urlcolor=blue,
linkcolor=blue,citecolor=blue]{hyperref}

\usepackage{amsmath,amsfonts,amssymb,amsthm}
\usepackage{amsthm, amssymb, amsmath, amsfonts, mathrsfs}

\usepackage{mathrsfs}
\usepackage{MnSymbol}
\usepackage{scalerel} 

\usepackage{color}
\usepackage{accents}

\usepackage{bbm}

\definecolor{labelkey}{gray}{.8}
\definecolor{refkey}{gray}{.8}

\definecolor{darkred}{rgb}{0.9,0.1,0.1}
\definecolor{darkgreen}{rgb}{0,0.5,0}

\newtheorem{theorem}{Theorem}[section]
\newtheorem{lemma}[theorem]{Lemma}

\newtheorem{corollary}[theorem]{Corollary}
\newtheorem{proposition}[theorem]{Proposition}

\theoremstyle{remark}

\renewenvironment{proof}[1][Proof]{ {\itshape \noindent {#1.}} }{$\Box$
\medskip}

\numberwithin{equation}{section}
\newcommand{\R}{\mathbb{R}}

\newcommand{\Z}{\mathbb{Z}}

\newcommand{\E}{\mathbb{E}}

\newcommand{\al}{\alpha}

\newcommand{\B}{\mathcal{B}}
\newcommand{\bbR}{\mathbb{R}}

\newcommand{\Q}{\mathcal{Q}}

\newcommand{\cal}{\mathcal}

\newcommand{\la}{\langle}
\newcommand{\ra}{\rangle}

\newcommand{\bT}{\mathbb{T}}

\newcommand{\bx}{\mathbf{x}}

\begin{document}

\title[Directed polymer on cylinder]{High temperature behaviors of the directed polymer on a   cylinder}
\author{\'Eric Brunet, Yu Gu, Tomasz Komorowski}

\address[\'Eric Brunet]{Laboratoire de Physique de l’\'Ecole normale sup\'erieure, ENS, Universit\'e PSL,
CNRS, Sorbonne Universit\'e, Universit\'e Paris Cit\'e, F-75005 Paris, France}

\address[Yu Gu]{Department of Mathematics, University of Maryland, College Park, MD 20742, USA}

\address[Tomasz Komorowski]{Institute of Mathematics, Polish Academy
  of Sciences, ul. \'{S}niadeckich 8, 00-656, Warsaw, Poland}

\maketitle

\begin{abstract}
In this paper, we study the free energy of the directed polymer on a
cylinder of radius $L$ with the inverse temperature $\beta$. Assuming
the random environment is given by a Gaussian process that is white in
time and smooth in space, with an arbitrary compactly supported spatial covariance function, we obtain precise scaling behaviors of the limiting free energy for high temperatures $\beta\ll1$, followed by large $L\gg1$, in all dimensions.  Our approach is based on a perturbative expansion of the PDE hierarchy satisfied by the multipoint correlation function of the polymer endpoint distribution. For the random environment given by  the $1+1$ spacetime white noise, we derive an explicit expression of the limiting free energy, confirming the result obtained through the replica method in \cite{brunet2000probability}.
\bigskip




\end{abstract}
\maketitle

\section{Introduction}

\subsection{Main result}

The random polymer model studied in this paper is associated with the
following stochastic heat equation (SHE) on the $d$-dimensional  torus $\bT_L^d$
of size $L>0$, i.e.,
{\begin{equation}\label{e.she}
\partial_t u=\frac12\Delta u+\beta uV(t,x), \quad\quad t>0, \,x\in \bT_L^d.
\end{equation}
The $d$ dimensional torus is the product
of $d$ copies of 
$\bT_L$, understood as  the interval $[-\tfrac{L}{2},\tfrac{L}{2}]$
with   identified endpoints.
The random potential $V$ is a Gaussian noise that is white in time and smooth in the spatial variable, and we assume 
\begin{equation}\label{e.cov}
\E \left[V(t,x)V(s,y)\right]=\delta(t-s)R(x-y),\quad  
\,(t,x),\,(s,y)\in \bbR\times \bT_L^d.
\end{equation}
Throughout the paper we assume $R(\cdot) $ belongs to
$C^\infty_0(\bbR^d)$ - the space of    smooth and compactly
supported functions. It is a fixed
non-negative function that does not depend on the parameter $L$. We
consider the case when $L$ is so large that
the  support of $R(\cdot) $ is contained within $\bT_L^d$,  and it is normalized so that
$\int_{\bbR^d} R(z)dz=1$. The parameter
 $\beta>0$, referred to as the inverse of temperature,  controls the strength of
 the noise.} 

Suppose $u(0,x)=\delta(x)$, then $Z_t=\int_{\bT_L^d}u(t,x)dx$ is the partition function of a (point to line) directed polymer model, as can be seen from the Feynman-Kac representation
\begin{equation}\label{e.defZt}
Z_t=\E_B \exp\left\{\beta \int_0^t V(s,B_s)ds-\frac12\beta^2 R(0)t\right\},
\end{equation}
where $\E_B$ is the expectation with respect to the  {standard} Brownian motion
$B$ on $\bT_L^d$ starting from the origin {and independent of the
  noise $V(t,x)$}. It is well-known,  {see e.g. \cite[Theorem 2.5]{GK21},} that the following limit exists 
\begin{equation}\label{e.deffree}
\gamma_L(\beta)=\lim_{t\to\infty}\frac{1}{t}\log Z_t=\lim_{t\to\infty}\frac{1}{t}\E \log Z_t,
\end{equation}
and is the thermodynamic limit of the free energy of the directed
polymer. The $\gamma_L(\beta)$  depends on the particular choice of
the spatial covariance function $R(\cdot)$. Here we are interested in
extracting the universal behaviors in the high temperature regime of
$\beta\to0$, followed by $L\to\infty$. Define the Fourier
transform of $R(\cdot)$ by
\begin{equation}\label{e.defRn}
\hat{R}(\xi)
  {=\int_{\bbR^d}R(x)e^{-i2\pi \xi\cdot x}dx}\quad \mbox{ for any $\xi\in\R^d$}. 
\end{equation}
Since
$R(\cdot)$ is a covariance function, we have $\hat{R}(\xi)\ge0$ for
all $\xi$.
The Fourier
coefficients of the $L$-periodic version of $R$, with 
the  size $L$ of the torus satisfying ${\rm
  supp}\,R(\cdot)\subset[-L/2,L/2]^d$, are given by 
$\hat{R}\left(\frac{n}{L}\right)$, $ n \in \Z^d$.

Here is the main result of the paper:
\begin{theorem}\label{t.mainth}
Fix any $L>0$. Then,
\begin{equation}\label{e.gammaex}
\gamma_L(\beta)=\gamma_L^{(2)}\beta^2+\gamma_L^{(4)}\beta^4+O(\beta^6),\quad
\mbox{as $\beta\ll1$,}
\end{equation}
with 
\[
\gamma_L^{(2)}=-\frac{1}{2L^d}, \quad\quad \gamma_L^{(4)}=-\frac{1}{8\pi^2 L^{2d-2}}\sum_{0\neq n\in \Z^d} \frac{1}{|n|^2}\hat{R}^2\left(\frac{n}{L}\right).
\]
In addition,
\begin{equation}\label{e.limitgamma4}
\begin{aligned}
&\lim_{L\to\infty}\gamma_L^{(4)}= -\frac{1}{24}, & d=1,\\
&\lim_{L\to\infty} \frac{L^2}{\log L} \gamma_L^{(4)}= -\frac{1}{4\pi} , & d=2,\\
&\lim_{L\to\infty} L^d \gamma_L^{(4)}=-\frac{1}{8\pi^2}\int_{\R^d} |\xi|^{-2}\hat{R}^2(\xi)d\xi,&d\geq3.
\end{aligned}
\end{equation}
\end{theorem}


\subsection{Context}

The study of directed polymers in random environments is an active area in probability and  statistical physics. The interests are in the transversal displacements of the polymer endpoint, the fluctuations of the free energy, the localization behaviors of the sample paths etc.  We refer to the monograph \cite{comets2017directed} for a general introduction to the subject.  The partition function of the directed polymer is naturally connected to the heat equation with a random potential, through the Feynman-Kac representation as \eqref{e.defZt}. After the Hopf-Cole transformation, it is related to the Kardar-Parisi-Zhang (KPZ) equation, which is a default model for interface growth subjected to random perturbations, see the reviews \cite{corwin2012kardar,quastel2015one} on the recent developments on the 1+1 KPZ universality class.

Besides studying the directed polymers on the free space where the
sample paths are spread out without any constraint, there have been
many recent developments on understanding how the underlying geometry
or the boundary conditions affect the large scale behaviors of the
polymer measure and the associated SHE and KPZ problem, see
e.g. \cite{corwin2018open,parekh2019kpz,gerencser2019singular,corwin2021stationary,bryc2021markov,barraquand2021steady,barraquand2021kardar}
and the references therein. In this paper, we consider the   polymers
confined to a cylinder and study the high temperature behaviors
of the limiting free energy. Our study is partly motivated by the
results in
\cite{lacoin2010new,berger2017high,nakashima2014remark,nakashima2019free},
where the same problem has been considered in the whole space. If we
denote the  limiting free energy by $\gamma_\infty(\beta)$ in this
case, it has been shown in the aforementioned works, for a large class of discrete models and as $\beta\to0$,
\begin{equation}\label{e.inf24}
\frac{1}{\beta^4} \gamma_\infty(\beta)\to -\frac{1}{24},  \quad\quad \mbox{ in } d=1,
\end{equation}
and
\begin{equation}\label{e.infpi}
\beta^2\log \gamma_\infty(\beta)\to -\pi, \quad\quad \mbox{ in } d=2.
\end{equation}
The limiting constants $-\frac{1}{24}$ and $-\pi$ are universal as they do not depend on the specific distributions of the underlying random environment. 
Compare to the expansion in \eqref{e.gammaex}, we see some similarity in $d=1$.
It is worth emphasizing that the free energy defined in \eqref{e.deffree} is actually the difference between the quenched and annealed free energies considered in those works. This is only a matter of convention: if we define the partition function by $ \E_B e^{\beta\int_0^t V(s,B_s)ds}$, then the quenched free energy is $t^{-1}\log \E_B e^{\beta\int_0^t V(s,B_s)ds}$, and the annealed free energy is $\tfrac12\beta^2R(0)$, so their difference is precisely $t^{-1}\log Z_t$ with $Z_t$ defined as in \eqref{e.defZt}. It is well-known that the free energy is associated to the localization properties of the polymer paths, and is related to the overlap fraction of two replicas, see e.g. the discussion in \cite[Chapter 5 and 6]{comets2017directed}. Therefore, the study of $ \gamma_L(\beta)$  for small $\beta$ sheds light on the localization properties of the polymer paths in high temperature regimes. 

Another motivation comes from the replica method used to compute the free energy. In \cite{brunet2000probability,brunet2000ground}, the authors considered the same problem of directed polymers   on a cylinder. For the environment of a $1+1$ spacetime white noise, using the Bethe ansatz method, they derived the expansions of the ground state energy $E(n,\beta,L)$ of the Delta Bose gas in $d=1$
\[
\mathcal{H}_n=\frac12\sum_{i=1}^n \nabla_i^2+\beta^2\sum_{1\leq i<j\leq n}\delta(x_i-x_j).
\]
In \cite[Equation (49)]{brunet2000probability}, it says that
\[
E(n,\beta,L)=-\big(\frac{\beta^2}{2L}+\frac{\beta^4}{24}\big)n+c_2n^2+c_3n^3+\ldots,
\]
for some explicit $c_2,c_3,\ldots$.
If the replica method gives the right answer here, then the coefficient of the $O(n)$ term, which is $-\big(\frac{\beta^2}{2L}+\frac{\beta^4}{24}\big)$, 
should be $\gamma_L(\beta)$ in the case of $R(\cdot)=\delta(\cdot)$. We will show in Section~\ref{s.white} that it is indeed the case. Namely, the limiting free energy takes the simple form $-\big(\frac{\beta^2}{2L}+\frac{\beta^4}{24}\big)$ in the case of $d=1, R(\cdot)=\delta(\cdot)$ and for a torus of size $L$. For the same problem on the whole space when $L=\infty$, the replica method also leads to the correct answer \cite{kardar1987replica,chen2015precise,bertini1999long,amir2011probability,sasamoto2010exact,sasamoto2010one,ghosal2020lyapunov}, and the limiting free energy is $-\frac{\beta^4}{24}$. Thus, our result shows that the ``finite size correction'' is given by $-\frac{\beta^2}{2L}$, and this confirms the physics prediction in \cite{krug}. 
 At the end of \cite{brunet2000probability}, the authors mentioned that ``another interesting extension of the present work would be to consider more general correlations of the noise'' and ``one could try to extend the approach to higher dimension as the relation between the directed polymer problem and the quantum Hamiltonian is valid in any dimension''. Our work can be viewed as a preliminary step along this direction, in which we obtain the high temperature expansions of the limiting free energy, for general covariance functions and in all dimensions.

Our approach is based on a formula that relates $\gamma_L(\beta)$ with the replica overlap of the polymer measure. The idea is to  perform a semi-martingale decomposition of $\log Z_t$, see e.g. \cite[Chapter 5]{comets2017directed}. After taking the expectation, the only contribution to $t^{-1}\log Z_t$ comes from the drift and can be expressed as a time average of  the overlap fraction of two replicas. 
On the cylinder, the polymer endpoint distribution converges exponentially fast to the stationary distribution, see the proofs in \cite{GK21,rosati2019synchronization} and the related results for stochastic Burgers equation \cite{sinai1991two}. The overlap fraction of two replicas is simply related to the two-point correlation function of the stationary distribution. In this way, the limiting free energy can be written explicitly as an integral involving the two-point correlation function of the stationary distribution and the spatial covariance function of the random environment, see \eqref{e.freeen} below.

On the cylinder, the stationary distribution of the polymer endpoint is related to that of the  KPZ equation (modulo a constant) and  to the stochastic Burgers equation. It is well-known that for the $1+1$ spacetime white noise, the stationary distribution of the KPZ equation is the Brownian bridge \cite{bertini1997stochastic,funaki2015kpz,gubinelli2017kpz,hairer2018strong}. Using this connection and Yor's formula for the density of exponential functionals of
Brownian bridge \cite{yor1992some}, the limiting free energy can be written down explicitly in this case, see Proposition~\ref{p.white} in Section~\ref{s.white}. For the noise with a general covariance structure, which is the main interest of this paper, there are no explicit formulas of the invariant measure. We proceed in a different way, using a partial differential equation (PDE) hierarchy satisfied by the $n-$point correlation functions of the stationary distribution, see \eqref{e.hierarchy} below. The PDE hierarchy was derived in \cite{GH20} on the whole space, and it admits a stationary solution  on the cylinder. An asymptotic expansion in $\beta^2$ on the level of the hierarchy leads to the corresponding expansion of the limiting free energy. The approach is surprisingly simple, and we can actually obtain the expansion in $\beta^2$ up to any order, see the discussion in Section~\ref{s.extension} below.

The same approach does not apply to the problem on the whole space. As
$t$ goes to infinity, the polymer endpoint spreads to infinity, hence
there is no   equilibrium. Nevertheless, the replica overlap is
invariant under the shift of the polymer endpoint. By embedding the
endpoint distribution into an abstract space, which factors out the
spatial shift, significant progress has been made recently on
the localization properties of the endpoint distribution, see \cite{bates2020endpoint,bakhtin2020localization,broker2019localization}. In this case, the limiting free energy
can be expressed as the solution of a variational problem, generalizing \eqref{e.freeen} in a sense. 

We mention two recent papers  on a nonlinear version of \eqref{e.she} on torus \cite{khoshnevisan2020dissipation,khoshnevisan2021dissipation}, where the dissipation rate was studied, i.e., how fast $ u(t,x)$ decays to zero. Among other interesting results, a stronger version of \eqref{e.deffree} was estabilished in the linear setting, see \cite[Theorem 1.3]{khoshnevisan2021dissipation}.

The rest of the paper is organized as follows. In Section~\ref{s.hierarchy}, we prove some preliminary results on the endpoint distribution of the directed polymer and derive the PDE hierarchy satisfied by the $n-$point correlation functions. Section~\ref{s.proof} is devoted to the asymptotic analysis of the PDE hierarchy and the proof of the main theorem. In Section~\ref{s.white}, we discuss the case of the $1+1$ spacetime white noise and some further extensions. 

\subsection*{Acknowledgements}
Y.G. was partially supported by the NSF through DMS-1907928/2042384. 
 T.K. acknowledges the support of NCN grant 2020/37/B/ST1/00426. We thank Bernard Derrida for interesting comments on the draft.

\section{Endpoint distributions of directed polymers  and a PDE hierarchy}
\label{s.hierarchy}

With $u$ solving \eqref{e.she} started from the initial data, { that is
given by 
a  non-trivial Borel measure}, define 
\begin{equation}\label{e.defrho}
\rho(t,x)=\frac{u(t,x)}{\int_{\bT_L^d}u(t,x')dx'},
\end{equation}
which is the quenched density of the endpoint distribution of polymer
of length $t$. We emphasize that $\rho(t,x)$ actually depends on
$\beta$ and we have kept the dependence implicit in our
  notation. Since we are interested in the high temperature regime,
i.e.  $\beta\ll1$, throughout the rest of the paper we assume $\beta\in(0,1)$. 

We first prove some results on $\rho(t,\cdot)$ and its relation to the
free energy $\gamma_L(\beta)$, some of which were obtained in
\cite{GK21}.

Let  $\mathcal{M}_1(\bT^d_L)$ be the space of Borel probability
measures on $\bT^d_L$ and  $\Z_+$ be the set of non-negative
integers. Denote by $ D(\bT^d_L)$ and $ D_c(\bT^d_L)$  the respective spaces of
all Borel measurable and continuous densities on the torus
$\bT_L^d$.
\begin{proposition}\label{p.rho}
$\{\rho(t,\cdot)\}_{t\geq0}$ is an $\mathcal{M}_1(\bT^d_L)$-valued
Markov process. In fact, for any $t>0$ it  takes values in $D_c(\bT^d_L)$.
The process has a unique invariant measure $\pi$   that is supported
on  $D_c(\bT^d_L)$.

Let $\varrho$ be a $D_c(\bT^d_L)$-valued  random variable with the distribution given by $\pi$, then the free energy can be expressed as 
\begin{equation}\label{e.freeen}
\gamma_L(\beta)=-\frac12\beta^2\int_{(\bT_L^{d})^2} R(x-y) \E [ \varrho(x)\varrho(y)]dxdy.
\end{equation}
In addition, $\{\varrho(x): x\in \bT_L^d\}$ is a continuous
trajectory, stationary {random
field}. For any $n\in \Z_+$ we have
{ \begin{equation}\label{e.mmesti}
C_*(n,R,L):= \E \big[\sup_{x\in \bT_L^d}\varrho(x)^n\big] <+\infty.
 \end{equation}}
 \end{proposition}
\begin{proof}
Throughout the proof, we use $C$ to denote a generic constant that may
depend on $n,R(\cdot),L$, but not on $\beta\in(0,1)$, and may
change from line to line.

The fact that the Markov process $\{\rho(t,\cdot)\}_{t\geq0}$ has a
unique invariant measure that is supported on the space of positive, continuous
densities on $\bT_L^d$ has been proved in \cite[Theorem 2.3]{GK21}. 
The expression of the free energy in \eqref{e.freeen} was given in \cite[Equation (2.22)]{GK21}.

To show the stationarity of the field $\{\varrho(x): x\in\bT_L^d\}$,
we start from the flat initial data $u(0,x)\equiv1$ so that
$\rho(0,x)\equiv L^{-d}$. By \cite[Theorem 2.3]{GK21}, we have
$\rho(t,\cdot)\Rightarrow \varrho(\cdot)$ in distribution on
$C(\bT_L^d)$, as $t\to\infty$. Note that, for each fixed $t>0$, the
field $\{\rho(t,x): x\in\bT_L^d\}$ is stationary, and this in turn
implies the stationarity of $\varrho$.

Estimate \eqref{e.mmesti} is a consequence of   \cite[Lemma 4.9]{GK21}.
\end{proof}

Throughout the rest of the paper, we assume that $\varrho$ is sampled from $\pi$. For any $n\geq 1$ and $\bx_{1:n}=(x_1,\ldots,x_n)\in\big(\bT_L^d\big)^n$, define the $n-$point correlation function of the stationary random field $\varrho(x)$:
\begin{equation}\label{e.defQn}
Q_n(\bx_{1:n})=Q_n(x_1,\ldots,x_n)=\E\big[\prod_{j=1}^n \varrho(x_j)\big].
\end{equation}
By Proposition~\ref{p.rho}, we have
\begin{equation}\label{e.bdQn}
\sup_{\bx_{1:n}\in (\bT_L^d)^n} Q_n(\bx_{1:n}) \leq C(n,R,L),
\end{equation}
{and $Q_n(\cdot)$ is a continuous function jointly in all its variables on $(\bT_L^d)^n$}.
Again to simplify the notations, we have kept the dependence of $Q_n$ on $\beta$ and $L$ implicit. Also note that since $\varrho(\cdot)$ is stationary and $\int_{\bT_L^d}\varrho(x)dx=1$, we actually have $Q_1(x)\equiv L^{-d}$. 

Here is the main result of this section
\begin{proposition}\label{p.hierarchy}
For any $n\geq1$, $Q_n:(\bT_L^d)^n\to \R_+$ is a smooth function, and the sequence $\{Q_n\}_{n\geq 1}$ solves the following PDE hierarchy: for any $n\geq1$, 
\begin{equation}\label{e.hierarchy}
\begin{aligned}
\frac12&\Delta Q_n+\beta^2\sum_{1\leq i<j\leq n} R(x_i-x_j)Q_n\\
=&\beta^2n\int_{\bT_L^d}Q_{n+1}(\bx_{1:n},x_{n+1})\sum_{i=1}^n R(x_i-x_{n+1})dx_{n+1}\\
&-\beta^2\frac{n(n+1)}{2}\int_{\bT_L^d\times \bT_L^d} Q_{n+2}(\bx_{1:n},x_{n+1},x_{n+2})R(x_{n+1}-x_{n+2})dx_{n+1}dx_{n+2}.
\end{aligned}
\end{equation}
\end{proposition}

\begin{proof}
To show \eqref{e.hierarchy}, we make use of a dynamic version proved in \cite{GH20}. Define
\[
\Q_n(t,\bx_{1:n})=\E\big[\prod_{j=1}^n \rho(t,x_j)\big],
\]
with $\rho(t,x)$ given by \eqref{e.defrho}. Then \cite[Theorem 1.1]{GH20} shows that $\{\Q_n\}_{n\geq1}$ satisfies the following hierarchy: for any $n\geq1, T>0$, and $f\in C^\infty\big((\bT_L^d)^n\big)$, 
\begin{equation}\label{e.parahi}
\begin{aligned}
\la f,\Q_n(T)\ra=\la f,\Q_n(0)\ra+\int_0^T \la \frac12\Delta f,\Q_n(t)\ra dt+\beta^2\sum_{k=0}^2 \int_0^T \la f_k,\Q_{n+k}(t)\ra dt,
\end{aligned}
\end{equation}
where 
\[
\begin{aligned}
&f_0(\bx_{1:n})=f(\bx_{1:n})\sum_{1\leq i<j\leq n}R(x_i-x_j),\\
&f_1(\bx_{1:n},x_{n+1})=-nf(\bx_{1:n})\sum_{i=1}^n R(x_i-x_{n+1}),\\
&f_2(\bx_{1:n},x_{n+2})=\frac12n(n+1)f(\bx_{1:n})R(x_{n+1}-x_{n+2}),
\end{aligned}
\]
and the brackets $\la\cdot,\cdot\ra$ in \eqref{e.parahi} are the corresponding $L^2$ inner product. Assume $\rho(0,\cdot)$ is sampled from the invariant measure $\pi$, then we have 
\[
\Q_n(t,\bx_{1:n})=Q_n(\bx_{1:n}), \quad\quad \mbox{ for all }t\geq0, n\geq1,
\]
so \eqref{e.parahi} actually becomes 
\begin{equation}
\label{010910-21}
\la \frac12\Delta f,Q_n\ra+\beta^2 \sum_{k=0}^2 \la f_k, Q_{n+k}\ra=0,
\end{equation}
which is  the weak formulation of \eqref{e.hierarchy}.
Let 
\begin{align}
  \label{F}
&F(\bx_{1:n}):=\beta^2 \Big\{
n\int_{\bT_L^d}Q_{n+1}(\bx_{1:n},x_{n+1})\sum_{i=1}^n
R(x_i-x_{n+1})dx_{n+1}\notag\\
&
- \frac{n(n+1)}{2}\int_{\bT_L^d\times \bT_L^d}
                           Q_{n+2}(\bx_{1:n},x_{n+1},x_{n+2})R(x_{n+1}-x_{n+2})dx_{n+1}dx_{n+2}\\
  &
    -\sum_{1\leq i<j\leq n}R(x_i-x_j)
Q_n(\bx_{1:n})\Big\}.\notag
\end{align}
The function $F$ is continuous and, substituting $f\equiv 1$ into
\eqref{010910-21}, we conclude that
$$
\int_{(\bT_L^d)^n}F(\bx_{1:n}) d
\bx_{1:n}=0.
$$ Therefore the Poisson equation
$\frac12\Delta \tilde Q_n=F$ has a unique, up to a constant,
solution $\tilde Q_n$ that belongs to any  Sobolev
space  $W^{2,p}\big((\bT^d_L)^n\big)$, $p\in[1,+\infty)$ - consisting of
functions with two generalized derivatives that are $L^p$
integrable. The function $Q_n-\tilde Q_n$ is harmonic on
$(\bT^d_L)^n$, in the weak sense, therefore, by the Weyl lemma, see
e.g. \cite[Theorem 2.3.1, p. 42]{morrey}, it is harmonic in the strong sense.
As a result $Q_n\in W^{2,p}\big((\bT^d_L)^n\big)$ for any
$p\in[1,+\infty)$ and $n\ge1$. Hence also $F\in
W^{2,p}\big((\bT^d_L)^n\big)$. This, by an application of the apriori
estimates, allows us to conclude that in fact $Q_n\in W^{4,p}\big((\bT^d_L)^n\big)$  for any
$p\in[1,+\infty)$ and $n\ge1$.  Using a bootstrap argument,  we can conclude that 
$Q_n\in C^\infty\big((\bT^d_L)^n\big)$ for any
 $n\ge1$. Since $Q_n$ are smooth functions, we know that they are classical solutions to \eqref{e.hierarchy}, which completes the proof.
\end{proof}

The next result is on the sample path properties of $\varrho$. We first introduce some notation: let
\begin{align*}
 &
D_\infty(\bT^d_L):= D(\bT^d_L)\cap C^\infty(\bT^d_L)\\
&
=\{f:\bT_L^d\to\R: 0\leq f\in C^\infty(\bT_L^d), \int_{\bT_L^d}f(x)dx=1\}.
\end{align*}
\begin{corollary}
  The field
$\{\varrho(x):\,x\in\bT_L^d\}$ has smooth realizations, i.e. the invariant
measure $\pi$ is
supported on $D_\infty(\bT^d_L)$. In addition, for any   $n\in \Z_+$
and multiindex $\alpha=(\alpha_1,\ldots,\alpha_d)$ with
$\alpha_j\in\Z_{+},\, j=1,\ldots,d$, there exists a constant $C=C(n,\alpha,R,d,L)$ 
 such that 
 \begin{equation}\label{e.mmesti}
 \E \big[\sup_{x\in \bT_L^d}|\partial^\alpha\varrho(x)|^n\big] \leq C(n,\alpha,R,L).
 \end{equation}
  \end{corollary}
  \proof
Since the covariance functions $Q_n(\cdot)$ are smooth, \cite[Theorem
2.2.2, p. 27]{adler2010geometry} implies that  $\partial^\alpha\varrho(x)$ exists
for each $x$ in the $L^2(\pi)$ sense. An application of \cite[Theorem
1.4.1, p. 31]{kunita} allows us to conclude that in fact
the derivative field has a.s. continuous modification for each multiindex
$\alpha$. This proves the existence of smooth realizations of the
field $\{\varrho(x):\,x\in\bT_L^d\}$.

It remains to show \eqref{e.mmesti}.  Note that it suffices to prove a weaker statement: that  for any $n\in \Z_+$ and multiindex
$\alpha=(\alpha_1,\ldots,\alpha_d)$ with
$\alpha_j\in\Z_{+},j=1,\ldots,d$ we have
 \begin{equation}\label{e.mmesti1}
C_{*,w}(n,\alpha,R,L):=   \E
\big[|\partial^\alpha\varrho(x)|^n\big]<+\infty.
\end{equation}
The latter is a simple consequence of the existence of an appropriate
derivative of the function $Q_n$.  Note that by stationarity the right hand side does not depend on
$x\in \bT^d_L$.

To prove that \eqref{e.mmesti1} implies \eqref{e.mmesti}, observe
  that by the Sobolev embedding there exists a
deterministic constant $C$ such that
\begin{equation}\label{010810-21}
\sup_{x\in \bT^d_L}|\partial^\alpha\varrho(x)|^n\le
C\|\varrho\|_{W^{k,p}(\bT^d_L)}^n
\end{equation}
for all realizations of $\varrho(\cdot)$,
provided 
$k>d/p+|\al|$. Here $|\al|=\sum_{j=1}^d\al_j$ and
$\|\varrho\|_{W^{k,p}(\bT^d_L)}=\sum_{|\al|\le
  k}\|\partial^\alpha\varrho\|_{L^p(\bT_L^d)}$ is the Sobolev
norm. The estimate \eqref{e.mmesti} is then a consequence of
\eqref{010810-21} and \eqref{e.mmesti1}.
\qed


By \eqref{e.freeen}, the free energy can be expressed in terms of the two-point correlation $Q_2$:
\begin{equation}\label{e.gammaLbeta}
\gamma_L(\beta)=-\frac12\beta^2\int_{(\bT_L^{d})^2}R(x_1-x_2)Q_2(x_1,x_2)dx_1dx_2.
\end{equation}
Thus, to obtain the asymptotics of $\gamma_L(\beta)$ in the high temperature regime of $\beta\to0$, it reduces to studying the asymptotic behaviors of $Q_2$ as $\beta\to0$. For $\beta=0$, the polymer measure degenerates to the Wiener measure, and the Brownian motion on $\bT^d_L$ has the unique stationary distribution given by the uniform measure, in which case the $n-$point correlation function $Q_n$ equals to $L^{-nd}$. The following lemma provides preliminary estimates on the difference between $Q_n$ and $L^{-nd}$ for $\beta\ll1$.

Denote $L_0^2\big((\bT_L^d)^n\big)$ the space of square integrable functions on $(\bT_L^d)^n$ with zero mean, i.e.
\[
L_0^2\big((\bT_L^d)^n\big)=\left\{f\in L^2\big((\bT_L^d)^n\big): \int_{(\bT_L^d)^n}f(\bx_{1:n})
  d{\bx }_{1:n}=0\right\}.
\]
Define 
\begin{equation}\label{e.barQn}
\bar{Q}_n=Q_n-L^{-nd}.
\end{equation}
\begin{lemma}\label{l.errQn}
For any $n\geq2$, we have
\begin{equation}\label{e.err2}
\|\bar{Q}_n\|_{L^2((\bT_L^d)^n)}=O(\beta^2), \quad\mbox{ as }\beta\to0.
\end{equation}
\end{lemma}

\begin{proof}
For any $n$, we rewrite \eqref{e.hierarchy} as $\frac12\Delta Q_n=\beta^2F$, 
where $F$ is a smooth function in $L^2_0\big((\bT_L^d)^n\big)$, given by \eqref{F}. Since $\int_{(\bT_L^d)^n}Q_n
  d{\bx }_{1:n}=1$, we have $\bar{Q}_n\in L^2_0\big((\bT_L^d)^n\big)$, which implies\[
\bar{Q}_n=2\beta^2\Delta^{-1}F,
\]
where $\Delta^{-1}$ is the inverse of $\Delta$, which is a bounded
operator from $L^2_0\big((\bT_L^d)^n\big)$ to $L^2_0\big((\bT_L^d)^n\big)$. It remains to use the fact that $\|F\|_{L^2((\bT_L^d)^n)}$ is bounded uniformly in $\beta\in(0,1)$ to complete the proof.
\end{proof}

Recall that $\gamma_L(\beta)$ is related to $Q_2$ through \eqref{e.gammaLbeta} and we have assumed $\int_{\bT_L^d} R(x)dx=1$, the above lemma gives the leading order of $\gamma_L(\beta)$
\[
\gamma_L(\beta)+\frac{\beta^2}{2L^d} =O(\beta^4), \quad\quad \mbox{ as } \beta\to0.
\]

\section{Proof of Theorem~\ref{t.mainth}}
\label{s.proof}

We start from the equation satisfied by $Q_2$
\begin{equation}\label{e.q2}
\begin{aligned}
\frac12&\Delta Q_2+\beta^2 R(x_1-x_2)Q_2\\
=&2\beta^2\int_{\bT_L^d} Q_3(x_1,x_2,x_3)[R(x_1-x_3)+R(x_2-x_3)]dx_3\\
&-3\beta^2\int_{\bT_L^d\times \bT_L^d} Q_4(x_1,x_2,x_3,x_4)R(x_3-x_4)dx_3dx_4.
\end{aligned}
\end{equation}

In light of Lemma~\ref{l.errQn}, we rewrite the equation for $Q_2$ in
terms of $\bar{Q}_2$ (see \eqref{e.barQn}), stated in the following lemma
\begin{lemma}\label{l.eqbarq2}
We have
\begin{equation}\label{e.eqbarQ}
\frac12\Delta \bar{Q}_2=\frac{\beta^2}{L^{2d}}\big(\frac{1}{L^d}-R(x_1-x_2)\big)+\beta^2\mathcal{E}_\beta(x_1,x_2),
\end{equation}
where $\mathcal{E}_\beta(x_1,x_2)$ is a smooth function in $L_0^2(\bT_L^d\times \bT_L^d)$ such that 
\begin{equation}\label{e.bdEbeta}
\begin{aligned}
\|\mathcal{E}_\beta\|_{L^2(\bT_L^d\times \bT_L^d)}=O(\beta^2).
\end{aligned}
\end{equation}
\end{lemma}

\begin{proof}
It is straightforward to check that 
\[
\begin{aligned}
\mathcal{E}_\beta(x_1,x_2)=&2\int_{\bT_L^d}\bar{Q}_3(x_1,x_2,x_3)[R(x_1-x_3)+R(x_2-x_3)]dx_3\\
&-3\int_{\bT_L^d\times \bT_L^d} \bar{Q}_4(x_1,x_2,x_3,x_4)R(x_3-x_4)dx_3dx_4-R(x_1-x_2)\bar{Q}_2(x_1,x_2).
\end{aligned}
\]
Since $\int_{\bT_L^d} R(x)dx=1$, we have $\int_{\bT_L^d\times \bT_L^d}
\mathcal{E}_\beta dx_1 dx_2=0$ from the equation \eqref{e.eqbarQ}. Then we only need to invoke Lemma~\ref{l.errQn} to complete the proof.
\end{proof}


Define $g_L: \bT_L^d\times \bT_L^d\to\R$ as the unique solution in $L_0^2(\bT_L^d\times \bT_L^d)$ to 
\begin{equation}\label{e.calQ2}
\frac12\Delta g_L=\frac{1}{L^{2d}}\Big(\frac{1}{L^d}-R(x_1-x_2)\Big).
\end{equation}
We have $g_L(x_1,x_2)=G_L(x_1-x_2)$ with $G_L:\bT_L\to\R$ solving
\begin{equation}\label{e.eqG}
\Delta G_L(x)=\frac{1}{L^{2d}}\big(\frac{1}{L^d}-R(x)\big), \quad\quad x\in \bT_L^d.
\end{equation}

Using $g_L$, we can refine Lemma~\ref{l.errQn} when $n=2$.
\begin{lemma}\label{l.err1}
We have 
\[
\|\bar{Q}_2-\beta^2g_L\|_{L^2(\bT_L^d\times \bT_L^d)}=O(\beta^4), \quad\quad \mbox{ as } \beta\to0.
\]
\end{lemma}
\begin{proof}
By \eqref{e.eqbarQ} and \eqref{e.calQ2}, we know that $f=\bar{Q}_2-\beta^2 g_L$ is the solution to 
\[
\frac12\Delta f=\beta^2 \mathcal{E}_\beta,
\]
with $\int_{\bT_L^d\times \bT_L^d} fdx_1dx_2=0$. Using
\eqref{e.bdEbeta}, we can apply the same argument as for Lemma~\ref{l.errQn} to complete the proof.
\end{proof}

Now we can complete the proof of the main theorem.

\begin{proof}[Proof of Theorem~\ref{t.mainth}]
Recall that 
\[
\gamma_L(\beta)=-\frac12\beta^2\int_{\bT_L^d\times \bT_L^d} R(x_1-x_2)Q_2(x_1,x_2)dx_1dx_2.
\]
We can rewrite the above in the form 
\[
\begin{aligned}
\gamma_L(\beta)=&-\frac12\beta^2\int_{\bT_L^d\times \bT_L^d} R(x_1-x_2)L^{-2d}dx_1dx_2-\frac12\beta^4\int_{\bT_L^d\times \bT_L^d}R(x_1-x_2)g_L(x_1,x_2)dx_1dx_2\\
&-\frac12\beta^4\int_{\bT_L^d\times \bT_L^d} R(x_1-x_2)(\beta^{-2}\bar{Q}_2(x_1,x_2)-g_L(x_1,x_2))dx_1dx_2.
\end{aligned}
\]
Applying Lemma~\ref{l.err1}  and using the fact that 
\[
\begin{aligned}
\int_{\bT_L^d\times \bT_L^d}R(x_1-x_2)g_L(x_1,x_2)dx_1dx_2=L^d\int_{\bT_L^d} R(x)G_L(x)dx,
\end{aligned}
\]
we immediately derive that 
\[
\gamma_L(\beta)=-\frac{\beta^2}{2L^d}-\frac12\beta^4L^d\int_{\bT_L^d}R(x)G_L(x)dx+O(\beta^6).
\]
By \eqref{e.eqG}, we have that the Fourier coefficients of $G$ equal
\[
\hat{G}_L(n)=\int_{\bT_L^d} G_L(x)e^{-i2\pi n\cdot
  x/L}dx=\frac{\hat{R}(n/L)}{4\pi^2 |n|^2L^{2d-2}},\quad n\not=0,
\]
with $\hat{R}(\cdot)$ defined in \eqref{e.defRn}, and $\hat{G}_L(0)=0$. By the Parseval identity, we obtain
\[
\int_{\bT_L^d}R(x)G_L(x)dx=\frac{1}{L^d} \sum_{n\neq 0} \hat{G}_L(n)\hat{R}(n/L)=\sum_{n\neq 0} \frac{\hat{R}(n/L)^2}{4\pi^2 |n|^2 L^{3d-2}}
\]
which completes the proof of \eqref{e.gammaex}. 


In $d=1$, $\hat{R}(\frac{n}{L})\to \hat{R}(0)=1$, as $L\to\infty$, which implies 
\[
\gamma_L^{(4)}=-\frac{1}{8\pi^2}\sum_{n\neq 0} |n|^{-2}\hat{R}(\frac{n}{L})^2\to -\frac{1}{8\pi^2  }\sum_{n\neq 0} |n|^{-2} =-\frac{1}{24}.
\]

In $d=2$,   we divide the summation into two parts:
\[
\gamma_L^{(4)}=-\frac{1}{8\pi^2 L^{2}}\left(\sum_{0\neq |n|<\delta L}+\sum_{|n|\geq \delta L}\right) |n|^{-2}\hat{R}\Big(\frac{n}{L}\Big)^2=:A_1+A_2,
\]
where $\delta>0$ is a   constant to be sent to zero after sending $L\to\infty$. For the second part, we have 
\[
L^2A_2=-\frac{1}{8\pi^2 L^2}\sum_{|n/L|>\delta}
\Big|\frac{n}{L}\Big|^{-2} \hat{R}\Big(\frac{n}{L}\Big)^2\to
-\frac{1}{8\pi^2}\int_{|\xi|>\delta} |\xi|^{-2}\hat{R}(\xi)d\xi,
\quad\mbox{as $L\to\infty$}.
\]
 For the first part, we write it as 
\[
L^2A_1=-\frac{1}{8\pi^2}\sum_{0\neq |n|<\delta L}|n|^{-2}+\frac{1}{8\pi^2} \sum_{|n|<\delta L}|n|^{-2} \left(1-\hat{R}\Big(\frac{n}{L}\Big)^2\right)=:C_1+C_2.
\]
Since $R$ is smooth and $\hat{R}(0)=1$, we have 
\[
\lim_{\delta\to0}\limsup_{L\to\infty}\frac{C_2}{\log L} =0.
\]
For $C_1$, by an elementary calculation we have 
\[
\lim_{\delta\to0}\lim_{L\to\infty}\frac{C_1}{\log L} = -\frac{1}{4\pi},
\]
and this completes the proof of the case in $d=2$.

In $d\geq3$, we have
\[
L^d \gamma_L^{(4)}=-\frac{1}{8\pi^2 L^{d}}\sum_{n\neq 0} \Big|\frac{n}{L}\Big|^{-2}\hat{R}\Big (\frac{n}{L}\Big)^2\to -\frac{1}{8\pi^2}\int_{\R^d} |\xi|^{-2}\hat{R}(\xi)^2d\xi, \quad\mbox{as $L\to\infty$}.
\]
The proof is complete.
\end{proof}

\section{$1+1$ spacetime white noise}
\label{s.white}

\subsection{Spacetime white noise}
In this section, we consider the case when $d=1$ and the  random
potential is a $1+1$ spacetime white noise, in which case
$R(\cdot)=\delta(\cdot)$. Define $h(t,x)=\beta^{-1}\log u(t,x)$, which is the formal
  solution to the KPZ equation, see \eqref{e.kpz} below.
By the results in \cite{bertini1997stochastic,funaki2015kpz,gubinelli2017kpz,hairer2018strong}, we know that the invariant measure for the process $\{h(t,x)-h(t,0):x\in\bT_L\}_{t\geq0}$ is given by the law of the Brownian bridge $\B_{0,L}(\cdot)$ with $\B_{0,L}(0)=\B_{0,L}(L)=0$. Here for the notational convenience, we extend the Brownian bridge periodically and also view it as a process on $\bT_L$.  For the polymer endpoint density, if we write it as
\[
\rho(t,x)=\frac{u(t,x)}{\int_{\bT_L}u(t,x')dx'}=\frac{e^{\beta(h(t,x)-h(t,0))}}{\int_{\bT_L}e^{\beta(h(t,x')-h(t,0))}dx'},
\]
it is immediate to conclude that the invariant measure is given by the
law of random densities on $\bT_L$
\[
\varrho(x)=\frac{e^{\beta \B_{0,L}(x)}}{\int_0^L e^{\beta \B_{0,L}(x')}dx'}.
\]
The random field $\{\varrho(x):x\in\bT_L\} $ is stationary,
therefore, the limiting free energy  in \eqref{e.freeen}
(with $R(\cdot)=\delta(\cdot)$) reduces to 
\begin{equation}\label{e.freewhite}
\begin{aligned}
\gamma_L(\beta)=&-\frac12\beta^2\int_{\bT_L^2} R(x-y) \E [ \varrho(x)\varrho(y)]dxdy\\
=&-\frac12\beta^2 \int_{\bT_L}\E[\varrho(x)^2]dx=-\frac12\beta^2L\E[\varrho(0)^2]\\
=&-\frac12\beta^2L \E (\int_0^L e^{\beta \B_{0,L}(x)}dx)^{-2}.
\end{aligned}
\end{equation}
The random variable   $\int_0^L e^{\beta \B_{0,L}(x)}dx$ appears
frequently in physics and mathematical finance, and we refer to
\cite{matsumoto2005exponential} for an extensive discussion. Its
density function can be written explicitly, see \cite[Proposition 6.2,
p. 527]{yor1992some}, using which we obtain the following proposition:
\begin{proposition}\label{p.white}
In the case of a $1+1$ spacetime white noise, we have 
\begin{equation}\label{e.freewhite1}
\begin{aligned}
\gamma_L(\beta)&=-\frac{\beta^6L}{4\pi}\exp\left\{ \frac{2\pi^2}{
  \beta^2L}\right\}  \int_0^{\infty}\frac{ (e^y-e^{-y})}{(e^{y/2}+e^{-y/2} )^6}\exp\left\{-\frac{2y^2}{ \beta^2L} 
  \right\}\sin\left(\frac{4\pi y}{\beta^2L}\right)dy\\
  &=-\frac{\beta^2}{2L}-\frac{\beta^4}{24}.
  \end{aligned}
\end{equation}
\end{proposition}

\begin{proof}
First, by the scaling property of the Brownian bridge
\[
\int_0^L e^{\beta \B_{0,L}(x)}dx\stackrel{\text{law}}{=}L \int_0^1 e^{\beta\sqrt{L}\B_{0,1}(x)}dx.
\]
  To simplify the notation, define $Y_\lambda=\int_0^1 e^{\lambda \B(x)}dx$, so it remains to compute $\E Y_\lambda^{-2}$. Denote the density of $Y_\lambda$ by $f_\lambda(z)$, by \cite[Proposition 6.2, p. 527]{yor1992some} we have 
\begin{equation}
\label{fla}
f_\lambda (z)=\frac{4}{\pi
 \lambda^2 z^2}\exp\left\{-\frac{4}{\lambda^2 z}+\frac{2\pi^2}{ \lambda^2}\right\}\int_0^{\infty}\exp\left\{-\frac{2y^2}{ \lambda^2}-\frac{4\cosh
  y}{\lambda^2 z}
  \right\}\big(\sinh y\big)\sin\left(\frac{4\pi y}{\lambda^2}\right)dy.
\end{equation}
Using the above density formula, we have \begin{align}
\label{031006-21}
&
\E Y_\lambda^{-2}=\int_0^\infty z^{-2} f_\lambda(z)dz=\frac{4}{\pi
 \lambda^2 }\exp\left\{\frac{2\pi^2}{ \lambda^2}\right\}\int_0^{\infty}\exp\left\{-\frac{4}{\lambda^2 z} \right\}\frac{1}{
  z^4}\\
&
\times \left( \int_0^{\infty}\exp\left\{-\frac{2y^2}{ \lambda^2}-\frac{4\cosh
  y}{\lambda^2z}
  \right\}\big(\sinh y\big)\sin\left(\frac{4\pi y}{\lambda^2}\right)dy\right)dz.\notag
\end{align}
Changing variables $z'=(z\lambda^2)^{-1}$ we get
\begin{align}
\label{031006-21b}
&
\E Y_\lambda^{-2}=\frac{4 \lambda^4}{\pi
 }\exp\left\{ \frac{2\pi^2}{
  \lambda^2}\right\}\int_0^{\infty}z^2\exp\left\{- 4 z \right\} \\
&
\times \left(\int_0^{\infty}\exp\left\{-\frac{2y^2}{ \lambda^2}-4z\cosh
  y 
  \right\}\big(\sinh y\big)\sin\left(\frac{4\pi y}{\lambda^2}\right)dy \right)dz.\notag
\end{align}
Note that
\begin{align*}
\int_0^{\infty}z^2\exp\left\{- 4 z(1+\cosh y) \right\} d z=\frac{1}{32(1+\cosh y)^3}.
\end{align*}
Hence, we get
\begin{align}
\label{031006-21d}
\E Y_\lambda^{-2}&=\frac{ \lambda^4}{8 \pi
 }\exp\left\{ \frac{2\pi^2}{
  \lambda^2}\right\}  \int_0^{\infty}\frac{\sinh y}{(1+\cosh y)^3}\exp\left\{-\frac{2y^2}{ \lambda^2} 
  \right\}\sin\left(\frac{4\pi y}{\lambda^2}\right)dy.
\end{align}

Recall that
\[
\gamma_L(\beta)=-\frac{\beta^2}{2L}\E Y_{\beta\sqrt{L}}^{-2},
\]
to complete the proof of \eqref{e.freewhite1}, it remains to show that
$\E Y_{\lambda}^{-2}=1+\frac{\lambda^2}{12}$. To see the latter note
that, it follows from \eqref{031006-21d} that
\begin{align*}
\E Y_\lambda^{-2}=\frac{ \lambda^4}{16 \pi
 } {\rm Im}\left(\int_{\bbR}\frac{\sinh y}{(1+\cosh
                   y)^3}\exp\left\{-\frac{2(y -\pi i)^2}{ \lambda^2} 
                   \right\} dy\right).                
\end{align*}
Using the relations $\sinh (y+i\pi)=-\sinh y$, $\cosh
(y+i\pi)=-\cosh y$ and the change of variables $y':=y-i\pi$, we can further rewrite the right hand side as being
equal to 
\begin{align*}
& \frac{ \lambda^4}{16 \pi
 } {\rm Im}\left(\int_{\bbR-i\pi}\frac{\sinh y}{(\cosh
                   y-1)^3}\exp\left\{-\frac{2y^2}{ \lambda^2} 
                 \right\} dy\right)\\
  &
    =-\frac{ \lambda^4}{16 \pi
 } {\rm Im}\left(\int_{\bbR+i\pi}\frac{\sinh y}{(\cosh
                   y-1)^3}\exp\left\{-\frac{2y^2}{ \lambda^2} 
                   \right\} dy\right).                
\end{align*}
The last equality follows from the change of variables $y':=-y$
and the fact the integrand is odd. We can write therefore that
\begin{align*}
\E Y_\lambda^{-2}=\frac{ \lambda^4}{32 \pi
 } {\rm Im}\left(\int_{\cal C}\frac{\sinh z}{(\cosh
                   y-1)^3}\exp\left\{-\frac{2z^2}{ \lambda^2} 
                   \right\} dz\right).           
\end{align*}
Here ${\cal C}$ is an arbitrary counter-clockwise oriented, contour surrounding the unique pole $z=0$ of
the holomorphic function $\varphi(z):=  \frac{\sinh z}{(\cosh
                   z-1)^3}\exp\left\{-\frac{2z^2}{ \lambda^2} 
                   \right\}$.
Hence
\begin{align*}
\E Y_\lambda^{-2}=\frac{ \lambda^4}{16 
 } {\rm Re}\,\Big({\rm res}\,\varphi\Big),          
\end{align*}
where ${\rm res}\,\varphi$ is the residuum of the function $\varphi$ at
$0$. To compute the latter, note that for $|z|$ sufficiently small, we
can write
\begin{align*}
&(\cosh
                 z-1)^{-3}=\left(\frac{z^2}{2}\right)^{-3}\left[1+\frac{z^2}{12}+\frac{z^4}{360}+z^6\psi(z)\right]^{-3}=
\frac{8}{z^6}   \left[1-\frac{z^2}{4} +
    \frac{z^4}{30}  +z^6\psi(z)\right].
\end{align*}
Here and below $\psi(z)$ is some generic function holomorphic in a
               neighborhood of $0$. We have also used the expansion $
\frac{1}{(1+z)^3}=1-3z+6z^2+z^3 \psi(z)
$ valid for $|z|<1$.
We also have
 \begin{align*}             
  &
    \sinh z=z\left[1+\frac{z^2}{6}+\frac{z^4}{120}+z^6 \psi(z)\right],\\
  &
  \exp\left\{-\frac{2z^2}{\lambda^2}\right\}=1-\frac{2z^2}{\lambda^2}+  \frac{2z^4}{\lambda^4}+z^6 \psi(z).
 \end{align*}
 Putting all things together we get
 \begin{align*}             
  &\varphi(z) 
      =\frac{8}{z^5}\left\{1+z^2\left(\frac{1}{6}-\frac{2}{\lambda^2}\right)+z^4\left(\frac{1}{120}
    -\frac{1}{3 \lambda^2} +  \frac{2}{\lambda^4}\right)+z^6 \psi(z)\right\}\\
  &
    \times\left[1-\frac{z^2}{4} +
    \frac{z^4}{30}  +z^6 \psi(z)\right].
 \end{align*}
 We conclude therefore that
 $$
 {\rm res}\,\varphi=8\left( 
    \frac{1}{6 \lambda^2} +
                 \frac{2}{\lambda^4}  \right)
$$
and
\begin{align*}
\E Y_\lambda^{-2}=\frac{\lambda^2}{12 } +1.    
\end{align*}
The proof is complete.
%
%
%
\end{proof}

\subsection{Asymptotic expansion of the invariant measure}
\label{s.extension}

From the proof of Theorem~\ref{t.mainth}, it is clear that the expansion in \eqref{e.gammaex} can be extended to an arbitrary high order. We only kept the first two terms since their expressions are more explicit. Our expansion is based on the  the two-point correlation function $Q_2$, because that is what the limiting free energy depends on. It actually corresponds to an expansion of the  invariant measure $\varrho$ in the parameter $\beta$. Below we sketch the heuristic connections.

Suppose that $d\ge 1$ and $R$ is a smooth and compactly supported function.
Define $h(t,x)=\frac{1}{\beta}\log u(t,x)$, which is the solution to the  KPZ equation
\begin{equation}\label{e.kpz}
\partial_th=\frac12\Delta h+\frac12\beta |\nabla h|^2+V-\frac12R(0)\beta.
\end{equation}
We can write the polymer endpoint distribution in terms of $h$ as 
\begin{equation}\label{e.rhoh}
\rho(t,x)=\frac{u(t,x)}{\int_{\bT_L^d}u(t,x')dx'}=\frac{e^{\beta h(t,x)}}{\int_{\bT_L^d}e^{\beta h(t,x')}dx'}=\frac{e^{\beta (h(t,x)-\bar{h}(t))}}{\int_{\bT_L^d}e^{\beta (h(t,x')-\bar{h}(t))}dx'}.
\end{equation}
Here $\bar{h}(t)=L^{-d}\int_{\bT_L^d}h(t,x)dx$ is the average of $h(t,x)$. Thus, an expansion of the stationary distribution of $h(t,\cdot)-\bar{h}(t)$ in $\beta$ would lead to a corresponding expansion of $\varrho$. For $\beta\ll1$, we approximate \eqref{e.kpz} by the Edwards-Wilkinson equation
\[
\partial_t h=\frac12\Delta h+V.
\]
There are no stationary invariant probability measures for the above equation on the torus, as a result of the growth of the zero mode. If we remove the zero mode and consider the following equation
\begin{equation}\label{e.tildeh}
\partial_t\tilde{h}= \frac12\Delta \tilde{h}+V-\bar{V},
\end{equation}
where $\bar{V}(t)=L^{-d} \int_{\bT_L^d}V(t,x)dx$, then as a Markov process it admits a
stationary distribution with an explicit density. Replacing $h(t,x)-\bar{h}(t)$ in
\eqref{e.rhoh} by the stationary solution $\tilde{h}$, i.e. the one
where the initial data is sampled from the invariant distribution, we
obtain the first order approximation of the stationary measure
$\varrho$. It is straightforward to check that the two-point
correlation function of the stationary solution to \eqref{e.tildeh}
is directly related to the solution to
\eqref{e.calQ2}. 

To make the above argument rigorous, one needs to control the error in the approximation of the KPZ equation by the Edwards-Wilkinson equation. For us, it seems more convenient to do it 
  on the level of $Q_2$ through the PDE hierarchy \eqref{e.hierarchy},
  where we may borrow
  analytic tools.


\end{document}